\documentclass[12pt]{amsart}
\usepackage{amscd,verbatim,amssymb}
\usepackage{color}

\newcommand{\F}{\mathbf{F}}

\newcommand{\Q}{\mathbf{Q}}
\newcommand{\Z}{\mathbf{Z}}

\newcommand{\fp}{\mathfrak{p}}
\newcommand{\fq}{\mathfrak{q}}

\newcommand{\cl}{\operatorname{cl}}

\newcommand{\inj}{\hookrightarrow}

\renewcommand{\lim}{\varprojlim}

\newcounter{spec}
{\end{list}}%

\newtheorem{thm}{Theorem}
\newtheorem{lemme}[thm]{Lemma}

\theoremstyle{definition}

\theoremstyle{remark}

\begin{document}

\title[On the generalised Tate conjecture]{On the generalised Tate conjecture for products of
elliptic curves over finite fields}
\author{Bruno Kahn}
\address{Institut de Math{\'e}matiques de Jussieu\\ UPMC - UFR 929, Ma\-th\'e\-ma\-ti\-ques\\ 4
Place Jussieu\\75005 Paris\\ France}
\email{kahn@math.jussieu.fr}
\date{Jan. 7, 2011}
\begin{abstract} We prove the generalised Tate conjecture for $H^3$ of products of elliptic
curves over finite fields, by slightly modifying the argument of M. Spiess \cite{spiess}
concerning the Tate conjecture. We prove it in full if the elliptic curves run among at most $3$
isogeny classes. We also show how things become more intricate from $H^4$ onwards, for more
that $3$ isogeny classes.
\end{abstract}

\maketitle

Let $\F_q$ be a finite field. It is known that the Tate conjecture for all smooth projective
varieties over $\F_q$ implies the generalised Tate conjecture for all smooth projective
varieties over $\F_q$  (\cite[Rk. 10.3 2)]{qjpam}, \cite[\S 1]{milne-ram}); however, the proofs
in these two references are non-effective. It is therefore of interest to ask if one can prove
the generalised Tate conjecture for certain explicit classes of $\F_q$-varieties. 

In \cite{spiess}, Michael Spiess proved the Tate conjecture for products of elliptic curves
over a finite field: this provides a natural candidate for such a class. In this note, we show
that a slight modification of his argument does yield the generalised Tate conjecture, in
cohomological degree $3$ or if the elliptic curves run over at most $3$ distinct isogeny
classes.

Contrary to \cite{qjpam} and \cite{milne-ram}, the proofs do not appeal to Honda's existence
theorem \cite{honda}. This theorem appears, however, when studying $H^4$ of a well-chosen product of $4$
elliptic curves: this is directly related to the delicate
combinatorics of Weil numbers\footnote{The corresponding computation seems in contradiction with
the one from \cite[Claim p. 130]{kowalski}.}; we illustrate the non-effectiveness of the
arguments from
\cite{qjpam} and \cite{milne-ram} in this case.

\begin{thm}\label{t1} Let $X$ be a product of elliptic curves over $\F_q$. Then the
generalised Tate conjecture holds for $H^3(\bar X,\Q_l)$: the subspace of Tate coniveau $1$
coincides with the first step of the coniveau filtration.
\end{thm}

 Let $q=p^r$ for $p$ be a prime number and
$r\ge 1$. As in \cite{spiess}, we write $[\rho]$ for
the ideal generated by an algebraic integer $\rho$. As in \cite[Def. 1]{spiess}, we also say that a Weil $q$-number $\alpha$ is
\emph{elliptic} if it arises from the Frobenius endomorphism of an elliptic curve over $\F_q$.
There are two kinds of elliptic Weil $q$-numbers: the supersingular ones, of the form $\pm
p^{r/2}$  and the ordinary ones, which generate a quadratic extension of
$\Q$ in which $p$ is totally decomposed. In the latter case, if $[p]=\fp_1\fp_2$, then 
\begin{equation}\label{eq2}
[\alpha]=\fp_1^r \text{ or } \fp_2^r
\end{equation}
 (compare \cite[Lemma 2]{spiess}.)

The main lemma is:

\begin{lemme}\label{l1} Let $\alpha_1,\alpha_2,\alpha_3$ be $3$ elliptic Weil $q$-numbers,
generating a multiquadratic number field $K/\Q$. Suppose that 
\[
[\alpha_1\alpha_2\alpha_3] =[q\beta] 
\]
with $\beta$ an algebraic integer. Then there exist $i\ne j$ such that
\[[\alpha_i\alpha_j]=[q].\]
\end{lemme}

\begin{proof} If two of the $\alpha_i$ are supersingular the assertion is obvious. Thus we may
assume that at least two of the $\alpha_i$ are ordinary. 

{\bf Case 1:} one of the $\alpha_i$, say $\alpha_3$, is supersingular. If
$[\alpha_1\alpha_2]\ne [q]$, one sees that $[\alpha_1\alpha_2]$ is not divisible by $[p]$.
(Using \eqref{eq2} as in \cite[proof of Lemma 3]{spiess},  either $\alpha_1$ and $\alpha_2$
generate the same quadratic field and then $[\alpha_1]=[\alpha_2]$, or $\alpha_1$ and
$\alpha_2$ generate a biquadratic extension $K/\Q$ in which $[p]=\fq_1\fq_2\fq_3\fq_4$ and then
without loss of generality, $[\alpha_1]=(\fq_1\fq_2)^r$ and $[\alpha_2]=(\fq_1\fq_3)^r$.) If
$r> 1$, we get a contradiction. If $r=1$, we have the equation
$[\alpha_1\alpha_2]=[\sqrt{p}\beta]$ in $K(\sqrt{p})$. Since $p$ is totally ramified in
$\Q(\sqrt{p})$, the prime divisors of $[p]$ in $K$ are totally ramified in $K(\sqrt{p})$ and we
get a new contradiction.

{\bf Case 2:} all the $\alpha_i$ are ordinary. We assume again that the conclusion of the
lemma is violated, and show that $[\alpha_1\alpha_2\alpha_3]$ is then not divisible by $[p]$. 

If (say) $\alpha_1$ and $\alpha_2$ generate the same quadratic field, then as seen in Case 1,
$[\alpha_1]=[\alpha_2]$ and $[\alpha_1\alpha_2\alpha_3]$ is not divisible by $[p]$. Suppose now
that  the $\alpha_i$ generate three distinct imaginary quadratic fields. In particular,
$[K:\Q]\ge 4$. If $[K:\Q]=4$, then $K=\Q(\alpha_1,\alpha_2)$ (say) and $\alpha_1$, $\alpha_2$
generate two distinct quadratic subextensions of $K$. Then $\alpha_3$ must generate the third
quadratic subextension: but this is impossible because the latter is real. Thus  $[K:\Q]=8$. 

We now set up some notation. Let $G=Gal(K/\Q)\simeq (\Z/2)^3$, and let $X(G)$ be the character
group of $G$. The quadratic subextensions generated by $\alpha_1,\alpha_2,\alpha_3$ correspond
to characters $\chi_1,\chi_2,\chi_3$ forming a basis of $X(G)$. Let
$(\sigma_1,\sigma_2,\sigma_3)$ be the dual basis of $G$. Finally, let $c\in G$ be the complex
conjugation: since $\chi_i(c)=1$ for all $i$, we find that $c=\sigma_1\sigma_2\sigma_3$. Note
that, since the $\alpha_i$ are Weil $q$-numbers, we have $\alpha_i\alpha_i^c=q$.

Since $p$ is totally decomposed in all $\Q(\alpha_i)$, it is totally decomposed in $K$. Pick a
prime divisor $\fp$ of $[p]$. We then have
\[[p]=\fp^{\sum_{\sigma\in G} \sigma}.\]

Since $\alpha_1$ is invariant under $\sigma_2$ and $\sigma_3$, we find from \eqref{eq2}, up to
changing $\alpha_1$ to $\alpha_1^c$:
\[[\alpha_1]=\fp^{r(1+\sigma_2)(1+\sigma_3)}\]
and similarly:
\[[\alpha_2]=\fp^{r(1+\sigma_1)(1+\sigma_3)},\quad
[\alpha_3]=\fp^{r(1+\sigma_1)(1+\sigma_2)}.\]

We now compute: $[\alpha_1\alpha_2\alpha_3]=\fp^{rm}$, with
\begin{multline*}
m=(1+\sigma_2)(1+\sigma_3)+(1+\sigma_1)(1+\sigma_3)+(1+\sigma_1)(1+\sigma_2)\\
=3+2(\sigma_1+\sigma_2+\sigma_3)+\sigma_2\sigma_3+\sigma_1\sigma_3+\sigma_1\sigma_2.
\end{multline*}

This shows that $\fp^{rm}$ is not divisible by $[p]$ (the summand $\sigma_1\sigma_2\sigma_3$
is missing). Similarly, $[\alpha_1\alpha_2\alpha_3^c]=\fp^{rm'}$ with
\begin{multline*}
m'=(1+\sigma_2)(1+\sigma_3)+(1+\sigma_1)(1+\sigma_3)+c(1+\sigma_1)(1+\sigma_2)\\
=2+\sigma_1+\sigma_2+3\sigma_3+2\sigma_1\sigma_3+2\sigma_2\sigma_3+\sigma_1\sigma_2\sigma_3
\end{multline*}
and $\fp^{rm'}$ is not divisible by $[p]$ (the summand $\sigma_1\sigma_2$ is missing). The
other possible products reduce to those by permutation of the $\alpha_i$ and conjugation by
$c$: the proof is complete.
\end{proof}

\begin{proof}[Proof of Theorem \ref{t1}]  
 It is sufficient to prove the equality after tensoring with a large enough number field $K$,
Galois over $\Q$. We first observe that the Frobenius action on $H^*(\bar X):=H^*(\bar
X,\Q_l)\otimes K$ is semi-simple since $X$ is an abelian variety (compare \cite[Lemma
1.9]{cell}). Let $v$ be an eigenvector of Frobenius, with eigenvalue $\rho$. Since $H^3(\bar
X)=\Lambda^3H^1(\bar X)$ and $X$ is a product of elliptic curves,  $v$ is a sum of vectors of the form $v_1\wedge v_2\wedge v_3$ where $v_i\in H^1(\bar
X)$ is an eigenvector with  Frobenius eigenvalue  $\alpha_i$ with $\alpha_1\alpha_2\alpha_3=\rho$, $\alpha_i$ corresponds to an elliptic curve
$E_i$ and $v_i$ comes from $H^1(\bar E_i)\inj H^1(\bar X)$.

Suppose $\rho$ is divisible by $q$. Without loss of generality, we may assume that $v$ is a single vector $v_1\wedge v_2\wedge v_3$. By Lemma \ref{l1}, up to renumbering we have
$[\alpha_1\alpha_2]=[q]$. As in \cite[Corollary p. 288]{spiess}, there is an integer $N\ge 1$
such that
$(\alpha_1\alpha_2)^N=q^N$. 

 By the Tate conjecture in codimension $1$ for $E_1\times E_2$ (Deuring, cf. Tate
\cite{Tate1}), $v_1\wedge v_2\otimes \Q_l(1)\in H^2(\bar E_1\times \bar E_2)(1)$ is of the form
$\cl(\gamma)$ where $\gamma$ is a cycle of codimension $1$ on $\bar E_1\times \bar E_2$ and
$\cl$ is the cycle class map. Hence $v\otimes \Q_l(1)=\cl(\pi^* \gamma)\cdot v_3$, with
$\pi:X\to E_1\times E_2$ the projection. 
\end{proof}

\begin{thm} \label{t2} Let $X$ be a product of elliptic curves, belonging to at most $3$
distinct isogeny classes. Then the generalised Tate conjecture holds for $X$ in all degrees and all coniveaux.
\end{thm}

The proof is a variant of the one above:  in the proof of Lemma \ref{l1}, Case 2, the
computation showing that $[\alpha_1\alpha_2\alpha_3]$ and $[\alpha_1\alpha_2\alpha_3^c]$ are
not divisible by $[p]$ extends to show that $[\alpha_1^{n_1}\alpha_2^{n_2}\alpha_3^{n_3}]$ and
$[\alpha_1^{n_1}\alpha_2^{n_2}(\alpha_3^c)^{n_3}]$ are not divisible by $[p]$ for any
nonnegative  integers $n_1,n_2,n_3$. This generalises Lemma \ref{l1} to any product of Weil
$q$-numbers involved in the cohomology of $X$. \qed

\bigskip

Finally, we show what problems arise when one tries to replace $3$ by $4$ in Theorem \ref{t1}
or \ref{t2}. Start again with three non isogenous ordinary elliptic curves $E_1,E_2,E_3$, with
Weil numbers $\alpha_1,\alpha_2,\alpha_3$. We retain the notation from Case 2 in the proof of
Lemma \ref{l1}. Apart from $\chi_1,\chi_2$ and $\chi_3$,
\[\chi_1\chi_2\chi_3\]
is the unique character which does not vanish on $c$. In the corresponding quadratic subfield
of $K$, there is the possibility of a new Weil $q$-number $\alpha_4$ with
\[[\alpha_4] = \fp^{r(1+\sigma_1\sigma_2)(1+\sigma_1\sigma_3)}.\]

This can actually be achieved provided $r$ is large enough. Since the class group $Cl(O_K)$ is
finite, we may choose $r$ such that $\fp^r$ is principal, say $\fp^r=[\lambda]$. Then
$N_{K/\Q}(\lambda)=q$ (since $K$ is totally imaginary) and we choose
$\alpha_4=\lambda^{(1+\sigma_1\sigma_2)(1+\sigma_1\sigma_3)}$.

Up to increasing $r$, we may assume that the similar formulas hold for $\alpha_1,\alpha_2$ and
$\alpha_3$.

By Honda's theorem \cite{honda}, $\alpha_4$ corresponds to a 4th (isogeny class of) elliptic curve $E_4$. Now 
$\alpha_1\alpha_2\alpha_3\alpha_4^c= \lambda^{m''}$ with
\begin{multline*}
m''= m+ c(1+\sigma_1\sigma_2)(1+\sigma_1\sigma_3)\\
= 3+2(\sigma_1+\sigma_2+\sigma_3)+\sigma_2\sigma_3+\sigma_1\sigma_3+\sigma_1\sigma_2\\ + \sigma_1\sigma_2\sigma_3(1+\sigma_1\sigma_2+\sigma_1\sigma_3+\sigma_2\sigma_3)\\
= N+2(1+\sigma_1+\sigma_2+\sigma_3)
\end{multline*}
with $N=\sum_{\sigma\in G} \sigma$. Thus $\alpha_1\alpha_2\alpha_3\alpha_4^c = q\beta^2$, with
\[\beta=\lambda^{(1+\sigma_1+\sigma_2+\sigma_3)}.\]

This $\beta$ is a new Weil $q$-number; it generates $K$ since the isotropy group of $[\beta]$
in $G$ is trivial. By the Honda-Tate theorem, it corresponds to the isogeny class of a simple
$\F_q$-abelian variety $A$ of dimension $4$ (see \cite[p. 142 formula (7)]{Tate1}).

Let us say that a Weil $q$-number $\gamma$ is \emph{ordinary} if $\gcd(\gamma,\gamma^c)=1$.
This is equivalent to requiring that $\gcd(p,\gamma+\gamma^c)=1$, hence, by \cite[Prop.
7.1]{waterhouse}, that the corresponding abelian variety be ordinary. Let $\gamma\in K$ be an
ordinary Weil $q$-number. Since $\gamma\gamma^c=q$, the divisor of $\gamma$ is of the form
$\fp^{rm_\gamma}$, where $m_\gamma\in \Z[G]$ is the sum of elements in a section of the
projection $G\to G/\langle c\rangle$. These sections form a torsor under the group of maps from
$G/\langle c\rangle$ to $\langle c\rangle$, so there are 16 of them. Up to conjugation by $c$,
we get $8$. Among these $8$, $4$ are given by the kernels of the characters
$\chi_1,\chi_2,\chi_3$ and $\chi_1\chi_2\chi_3$, recovering $\alpha_1,\alpha_2,\alpha_3$ and
$\alpha_4$. Among the $4$ remaining ones, there is the one defining $\beta$; since the isotropy
group of $[\beta]$ is trivial, the other ones are conjugate to it. We have exhausted the
ordinary Weil $q$-numbers contained in $K$.

Let $X=\prod_{i=1}^4  E_i$. If we run the technique of proof of \cite{qjpam} or
\cite{milne-ram} to try and prove the generalised Tate conjecture for $N^1H^4(\bar X)$, we end
up with a Tate cycle in $H^6(\bar X\times \bar A)(3)$. This Tate cycle is exotic in the sense
that it is not a linear combination of products of Tate cycles of degree $2$ (cf. \cite[p.
136]{milne}), because the relation 
\[\alpha_1\alpha_2\alpha_3\alpha_4^c (\beta^2)^c= q^3\]
cannot be reduced to relations of degree $2$. I have no idea if the Tate conjecture can be
proven for $X\times A$. Can the methods of \cite{milne} be used to answer this question? 

\subsection*{Acknowledgements} Most of this work was done during a stay at IMPA (Rio de
Janeiro) in November 2010, in the framework of the France-Brazil cooperation. I thank the first
for its hospitality and the second for its support.

\end{document}